\documentclass[graybox]{svmult}

\usepackage{mathptmx}       % selects Times Roman as basic font
\usepackage{helvet}         % selects Helvetica as sans-serif font
\usepackage{courier}        % selects Courier as typewriter font
\usepackage{type1cm}        % activate if the above 3 fonts are
\usepackage{makeidx}         % allows index generation
\usepackage{graphicx}        % standard LaTeX graphics tool
\usepackage{multicol}        % used for the two-column index
\usepackage[bottom]{footmisc}% places footnotes at page bottom

\makeindex             % used for the subject index
\def\ve{\varepsilon}

\def\C{{\cal C}}

\def\R{I\!\!R}
\def\M{{\cal M}}
\def\N{{\cal N}}
\def\V{{\cal V}}
\def\W{{\cal W}}

\begin{document}

\title*{Lyapunov's theorem via Baire category}
\author{Marco Mazzola and Khai T. Nguyen}
\institute{Marco Mazzola \at Sorbonne Universit\'e, Institut de Math\'ematiques de Jussieu - Paris Rive Gauche, CNRS,
%UMR 7586, CNRS,   Univ Paris Diderot, Sorbonne Paris Cit\'e,  \newline\indent
Case 247, 4 Place Jussieu, 75252 Paris, France
\email{marco.mazzola@imj-prg.fr}
\and Khai T. Nguyen \at Department of Mathematics, North Carolina State University \email{khai@math.ncsu.edu}}
\maketitle

\abstract*{
%Each chapter should be preceded by an abstract (10--15 lines long) that summarizes the content. The abstract will appear \textit{online} at \url{www.SpringerLink.com} and be available with unrestricted access. This allows unregistered users to read the abstract as a teaser for the complete chapter. As a general rule the abstracts will not appear in the printed version of your book unless it is the style of your particular book or that of the series to which your book belongs.
%Please use the 'starred' version of the new Springer \texttt{abstract} command for typesetting the text of the online abstracts (cf. source file of this chapter template \texttt{abstract}) and include them with the source files of your manuscript. Use the plain \texttt{abstract} command if the abstract is also to appear in the printed version of the book.
}

\abstract{
Lyapunov's theorem is a classical result in convex analysis, concerning the convexity of the range of nonatomic measures. Given a family of integrable vector functions on a compact set, this theorem allows to prove the equivalence between the range of integral values obtained considering all possible set decompositions and all possible convex combinations of the elements of the family. Lyapunov type results have several applications in optimal control theory: they are used to prove bang-bang properties and existence results without convexity assumptions. Here, we use the dual approach to the Baire category method in order to provide a ``quantitative'' version of such kind of results applied to a countable family of integrable functions.
}
\section{Introduction}
The use of Baire categories in the analysis of nonconvex differential inclusions started with the seminal paper by A. Cellina \cite{Cellina}. These methods were later developed and adapted to various problems involving nonconvex ordinary and partial differential inclusions, notably in a series of articles by F. S. De Blasi and G. Pianigiani (see e.g. \cite{DeBlPian} and the bibliography therein). It is now known, for example, that the set $S^{ext}$ of extremal solutions of a differential inclusion, associated to a Lipschitz continuous multifunction with nonempty, compact and convex images, is residual in the set of all solutions $S$, i.e. it contains the intersection of countably many open dense subsets of $S$.

The same problem has been more recently approached by A. Bressan \cite{B2} from a ``dual" point of view. The procedure is the following: introduce auxiliary functions $v$ belonging to some complete space $V$; associate to each $v\in V$ a nonempty subset $S^v\subseteq S$; finally, show that the set of functions $v\in V$ satisfying $S^v\subseteq S^{ext}$ is residual in $V$. An advantage of this approach with respect to the ``direct'' one is that it could work even in the case when $S^{ext}$ is not dense in $S$. For the differential inclusion problem mentioned above, this situation can appear when no Lipschitzianity assumptions are imposed on the multifunction.

The dual approach was employed in \cite{BMN} in order to derive an extension of the classical bang-bang theorem in linear control theory. In very broad terms, it was proved that for almost every $v$ in a space of auxiliary functionals, there is a unique control minimizing $v$ and steering the system between two given points; furthermore, this control arc takes values almost everywhere within the extremal points of the set of admissible controls.
The classical proof of the bang-bang principle is actually based on a Lyapunov type theorem (see \cite{Cesari}). This result can be stated as follows.
Consider a finite family of Lebesgue integrable functions $f_1,\ldots,f_m$ from a compact subset $K\subset \R^d$ to $\R^n$ and the simplex of $\R^m$
$$\Delta_m~\doteq~\left\{\zeta=(\zeta_1,\ldots,\zeta_m)\in\R^m~\Big|~\zeta_i\geq0\ \ \forall\,i=1,\ldots,m,\ \sum_{i=1}^m\zeta_i=1\right\}\,.$$
Denote by $\M(K,\Delta_m)$ the set of Lebesgue measurable functions from $K$ to $\Delta_m$. Then, for any $\theta=(\theta_1,\ldots,\theta_m)\in\M(K,\Delta_m)$ there exists a measurable partition $\{E_1,\ldots,E_m\}$ of $K$ such that
$$
\int_{E_1} f_1(x)\, dx+\ldots+\int_{E_m}  f_m(x)\, dx~=~\sum_{i=1}^{m}\int_{K}\theta_i(x) f_i(x)~dx\,.
$$
An alternative ``extremal'' formulation of this theorem is the following. Given $\bar\theta=(\bar\theta_1,\ldots,\bar\theta_m)\in\M(K,\Delta_m)$, denote
$$\alpha~\doteq~\int_K \bar\theta_1(x)\, f_1(x)\, dx+\ldots+\int_K \bar\theta_m(x)\, f_m(x)\, dx~\in~\R^n\,.$$
Let $\Delta_{m}^{ext}$ be the set of extreme points of $\Delta_m$. According to Lyapunov's theorem, the set 
$$
\mathcal{A}^{ext}_{\alpha}~\doteq~\left\{\theta\in \M\left(K,\Delta_m^{ext}\right)~\Big|~\sum_{i=1}^{m}\int_{K}\theta_i(x) f_i(x)~dx~=~\alpha\right\}
$$
is nonempty. In the present paper, we aim to provide an alternative proof of this result based on the Baire category method, implying besides that $\mathcal{A}^{ext}_{\alpha}$ is actually residual in the set
$$\left\{\theta\in \M\left(K,\Delta_m\right)~\Big|~\sum_{i=1}^{m}\int_{K}\theta_i(x) f_i(x)~dx~=~\alpha\right\}$$
in a ``dual" sense.

The equivalence between the range of integral values obtained considering all possible set decompositions and all possible convex combinations of given vector functions plays an important role in optimal control theory, that goes beyond the application to the bang-bang theorem. For instance, it can be used to derive existence theorems for optimal control problems without convexity assumptions (see e.g. \cite{Angell,Suryana}). 

\section{A dual approach to Lyapunov's theorem}
For any continuous function $v:K\to\R^m$, consider the constrained optimization problem
\begin{equation}\label{P1}
\displaystyle \mathrm{Minimize}_{\theta\in\mathcal{A}_{\alpha}}~\int_{K}\theta(x)\cdot v(x)dx
\end{equation}
over the set
\begin{equation}\label{C1}
\mathcal{A}_{\alpha}~\doteq~\left\{\theta\in\M(K,\Delta_m)~\Big|~\sum_{i=1}^{m}\int_{K}\theta_i(x) f_i(x)~dx~=~\alpha\right\}\,,
\end{equation}
where $\theta(x)\cdot v(x)\doteq \sum_{i=1}^m\theta_i(x)v_i(x)$ denotes an inner product. It is clear that (\ref{P1})--(\ref{C1}) admits at least a solution. Indeed, since $\theta_m=1-\sum_{i=1}^{m-1}\theta_i$, the problem (\ref{P1})--(\ref{C1}) is equivalent to 
\begin{equation}\label{P1-1}
\displaystyle \mathrm{Minimize}_{\tilde{\theta}\in\mathcal{B}}~\int_{K}\sum_{i=1}^{m-1}\tilde{\theta}_i(x)\, (v_i(x)-v_m(x))dx
\end{equation}
over the set
\begin{eqnarray}\label{C1-1}
\mathcal{B}~\doteq~\Big\{\tilde{\theta}&\in&{\bf L}^{\infty}(K,\R^{m-1})~\Big|~\tilde{\theta}_i(x)\geq0\ \ \forall\,i=1,\ldots,m-1,\ \sum_{i=1}^{m-1}\tilde{\theta}_i(x)\leq1,\ a.e.\ x\in K,\nonumber\\
&&\qquad\quad\qquad\sum_{i=1}^{m-1}\int_{K}\tilde{\theta}_i(x) \big(f_i(x)-f_m(x)\big)~dx~=~\alpha-\int_{K}f_m(x)~dx\Big\}\,.
\end{eqnarray}
%where $\tilde{\theta}_i=\theta_i$, $\tilde{w}_i=v_i-v_n$, $\tilde{g}_i=f_i-f_n$, $\beta=\alpha-\int_{K}f_n(x)dx$ and 
%$$\tilde{\Delta}_{m-1}~\doteq~\left\{\zeta=(\zeta_1,\ldots,\zeta_{m-1})\in\R^{m-1}~\Big|~\zeta_i\geq0\ \ \forall\,i=1,\ldots,m-1,\ \sum_{i=1}^{m-1}\zeta_i\leq1\right\}\,.$$
Thanks  to Alaoglu's theorem, for every sequence $(\tilde{\theta}^n)_{n= 1}^\infty\subset \mathcal{B}$, there exists a subsequence $(\tilde{\theta}^{n_k})_{k= 1}^\infty$  converging weakly* to some $\tilde{\theta}\in {\bf L}^{\infty}(K,\R^{m-1})$ satisfying $\|\tilde{\theta}\|_{{\bf L}^{\infty}(K,\R^{m-1})}\leq 1$. Hence  
\begin{equation}\label{w-c}
\lim_{n_k\to+\infty}~\int_{K} \sum_{i=1}^{m-1}[\tilde{\theta}_i^{n_k}(x)-\tilde{\theta}_i(x)] w_i(x)~dx~=~0\qquad\forall \,w\in {\bf L}^1(K,\R^{m-1})\,
\end{equation}
yields
$$\sum_{i=1}^{m-1}\int_{K}\tilde{\theta}_i(x) \big(f_i(x)-f_m(x)\big)~dx~=~\alpha-\int_{K}f_m(x)~dx\,.
$$
Since $\sum_{i=1}^{m-1}\tilde{\theta}^{n_k}_i(x)\leq1$ for a.e. $x\in K$ and $\tilde{\theta}^{n_k}_i(x)\geq 0$ for a.e $x\in K$ and any $i\in \{1,2,...,m-1\}$,  by a contradiction argument  one obtains from (\ref{w-c}) that $\tilde{\theta}$ satisfies the same properties. Therefore, the set $\mathcal{B}$ is weakly*-compact in ${\bf L}^{\infty}(K,\R^{m-1})$ and it yields the existence of solutions to (\ref{P1-1})--(\ref{C1-1}).
\quad\\
\quad\\
Let's define 
\begin{equation}\label{W}
\mathcal{V}_{\alpha}~\doteq~\left\{ v\in\C(K,\R^m)~|~ (\ref{P1})-(\ref{C1})~has~a~unique~solution\right\}\,.
\end{equation}
Here, $\C(K,\R^m)$ is the space of continuous function on $K$ with values in $\R^m$. Our main result is stated as follows.
\begin{theorem}\label{MT1}
$\mathcal{V}_{\alpha}$ is a residual subset of $\mathcal{C}(K,\R^m)$, i.e. it contains the intersection of countably many open dense subsets of $\mathcal{C}(K,\R^m)$. Moreover, for any $v\in\mathcal{V}_{\alpha}$, the unique optimal solution $\theta^*$ takes values in $Ext\,(\Delta_m)$ almost everywhere in the compact set $K$.
\end{theorem}
The main ingredient in the proof of the above theorem is  the following lemma.
\begin{lemma}\label{lemma}
Let $g: K\to\R^n$ be a Lebesgue integrable function. Then the set $\mathcal{W}^{g}$ of continuous functions $w\in\mathcal{C}(K,\R)$ such that
\begin{equation}\label{meas0}
\hbox{meas}\Big(\bigl\{x\in K~|~ w(x)=\lambda\cdot g(x) \bigr\}\Big)~=~0\qquad\hbox{for all}~~\lambda\in\R^n
\end{equation}
is residual in $\C(K,\,\R)$.
\end{lemma}
\begin{proof}
For every positive integer $N$ and every 
$\ve>0$, call $\W^{g}_{\ve,N}$ the set of all $w\in \C(K,\,\R)$ such that
\begin{equation}\label{mease}
\hbox{meas}\Big(\big\{x\in K~|~ w(x)~=~\lambda \cdot g(x)
\big\}\Big)~<~\ve
\end{equation}
whenever $\lambda\in [-N,N]^n$.
The Lemma is proved once we show that, for every $\ve$ and $N$, $\W^{g}_{\ve,N}$ is open and dense in $\C(K;\,\R)$.
\quad\\
\quad\\
\noindent
{\bf 1.} We begin by proving that $\W^{g}_{\ve,N}$ is open.
Fix $w\in  \W^{g}_{\ve,N}$. For any $\lambda\in [-N,N]^n$, define
\begin{equation}\label{vel}
\ve_\lambda~\doteq~\ve-\hbox{meas}\Big(\bigl\{x\in K~|~ w(x)~=~\lambda \cdot g(x)\bigr\}
\Big)~>~0\,.
\end{equation}
Using  Lusin's theorem, there exists a continuous function $g_\lambda:K\mapsto\R^n$ such that
%  $|g_\lambda(x)|\leq 1$ for all $x\in K$ and 
\begin{equation}\label{El}
\hbox{meas}\Big(\bigl\{ x\in K~|~g_\lambda(x)\not= g(x)\bigr\}\Big)~<~\ve_\lambda/4\,.
\end{equation}
Consider the compact set of $\R^n$
$$
E_\lambda~\doteq~\big\{x\in K~|~ w(x)~=~\lambda\cdot g_\lambda(x)\big\}\,.
$$
By the regularity properties of Lebesgue measure,
there exists a relatively open set $O_\lambda\subset K$ such that 
\begin{equation}\label{Ol} E_\lambda~\subseteq~O_\lambda\quad\mathrm{and}\quad \hbox{meas} (O_\lambda \backslash E_\lambda)~
<~{\ve_\lambda \over  2}\,.
\end{equation}
By the continuity of $g_{\lambda}$ and $w$, one has
$$
\min_{x\in K\backslash O_\lambda} \Big|w(x)-\lambda\cdot g_\lambda(x)\Big|~\doteq~\delta_{\lambda}~>~0\,.
$$
For any function $\tilde w \in\mathcal{C}(K,\R)$ such that  
$$\|\tilde w-w\|_{\infty}~=~\sup_{x\in K}~|\tilde{w}(x)-w(x)|~<~r_\lambda~\doteq~{\delta_\lambda\over 3\max\{1,\|g_\lambda\|_{\infty}\}}\,,$$
it holds
$$
\Big|\tilde w(x)-\lambda\cdot g_\lambda(x)\Big|~>
~{2\over 3}\delta_\lambda\qquad\forall x\in K\setminus O_\lambda\,.
$$
In turn, if $|\tilde \lambda-\lambda|~<~r_\lambda$, this implies
\[
\Big|\tilde w(x)-\tilde\lambda\cdot g_\lambda(x)\Big|~>~{\delta_\lambda\over 3}~>~0\qquad\forall x\in K\backslash O_\lambda
\]
and it yields
\begin{equation}\label{del}
\hbox{meas}\Big(\bigl\{x\in K~|~ \tilde w(x)~=~\tilde\lambda\cdot g_\lambda(x)\bigr\}\Big)~\leq~\hbox{meas}\,(O_\lambda)\,.
\end{equation} 
By (\ref{vel}), (\ref{El}), (\ref{Ol}) and (\ref{del}),  if   
\begin{equation}\label{wtl}
\|\tilde w - w\|_{\infty}
~< ~{r_\lambda}\qquad\hbox{and} \qquad |\tilde \lambda-\lambda|~<~
{r_\lambda}\,,
\end{equation}
then it holds
\begin{eqnarray}\label{nep}
\hbox{meas}\Big(\bigl\{x\in K&|& \tilde w(x)~=~\tilde\lambda\cdot
g(x)\bigr\}\Big)\\\nonumber
~&<&~\hbox{meas}\Big(\bigl\{x\in K~|~ \tilde w(x)~=~\tilde\lambda\cdot
g_\lambda(x)\bigr\}\Big)+{\ve_\lambda\over4}\\\nonumber
&\leq&~\hbox{meas}\,(O_\lambda)+{\ve_\lambda\over4}~<~\hbox{meas}( E_\lambda)+{3\over4}\ve_\lambda\\\nonumber
 &<&~\hbox{meas}\Big(\bigl\{x\in K~|~ w(x)~=~\lambda\cdot
g(x)\bigr\}\Big)+{1\over4}\ve_\lambda+{3\over4}\ve_\lambda~=~\ve\,.
\end{eqnarray}
Repeating the above construction, for every $\lambda\in  [-N,N]^n$
there exists
 $r_\lambda>0$ so that the inequalities 
(\ref{wtl}) imply (\ref{nep}).
Since the set $[-N,N]^n$ is compact, we can select a finite family 
$\{\lambda^1,...,\lambda^M\} \subset [-N,N]^n$ such that the corresponding open balls 
$B\bigl(\lambda^k, r_{\lambda^k}\bigr)$ satisfy
$$
[-N,N]^n~\subset~\bigcup^M_{k=1}B\bigl(\lambda^k, r_{\lambda^k}\bigr)\,.
$$
Setting $r~\doteq~\min_{1\leq k\leq M} ~r_{\lambda^k}$, 
 for every $\tilde w\in B\bigl(w,r\bigr)$ and $\lambda\in [-N,N]^n$ we obtain
$$
\hbox{meas}\Big(\bigl\{x\in K~|~ \tilde w(x)~=~\lambda\cdot
g(x)\bigr\}\Big)~<~\ve\,.
$$
Therefore, $B\bigl(w,r\bigr)\subseteq \W^{g}_{\ve,N}$, proving that the set $\W^{g}_{\ve,N}$ is open in $\C(K,\,\R)$.
\quad\\
\quad\\
{\bf 2.} It remains to prove that each $\W^g_{\ve,N}$ is dense in $\C(K;\, \R)$.
Relying on Lusin's theorem, it is not restrictive to assume that 
$g$ is continuous. Given any $\eta>0$ and $\tilde w\in\C(K, \R)$, we will construct a function $w\in\W^g_{\ve,N}$, satisfying 
\begin{equation}\label{ball}
\|w-\tilde w\|_{\infty}~<~\eta\,.
\end{equation}
For simplicity,  without loss of generality we will assume that $K=[0,1]^d$. Let's choose an integer $m$ sufficiently large so that $m^d~\geq~n+1$ and $h~\doteq~{1\over m}$ satisfies
\begin{equation}\label{msmall}
h^d~<~{\ve\over 2n}
\end{equation}
and
\begin{equation}\label{ucont}
(x,x')\in K^2\,,\,|x-x'|~\leq~h\sqrt{d}\quad\Longrightarrow\quad\big|\tilde w(x)-\tilde w(x')\big|~<~{\eta\over2}\,.
\end{equation}
We adopt the following notation: a vector $y\in(\R^m)^d$ will be indexed by $y~=~(y_j)_{j\in\{0,\ldots,m-1\}^d}$. For every $\xi\in [0, h]^d\,$, $\lambda\in[-N,N]^n$ and $y\in(\R^m)^d$, define
$$
x_{j,\xi}~\doteq~ \xi + h\,j\,,\qquad\quad j\in\{0,\ldots,m-1\}^d
$$
and
\begin{equation}\label{Jxila}
J_{\lambda,\xi}(y)~\doteq~\left\{j\in\{0,\ldots,m-1\}^d~\Big|~y_j~=~\lambda\cdot g(x_{j,\xi})\right\}\,.
\end{equation}
We claim that the set
$$
Y(\xi)~\doteq~\left\{y\in(\R^m)^d~\Big|~\#\,J_{\lambda,\xi}(y)~\leq~n, ~~~\forall\,\lambda\in [-N,N]^n \right\}
$$
is dense in $(\R^m)^d$. Indeed, the complementary of $Y(\xi)$ is contained in the union of a finite family of proper hyperspaces: for every collection of indexes
$$J~=~\{j_1,\ldots,j_{n+1}\}~\subset~\{0,\ldots,m-1\}^d\,,$$
let us define the projection
$$
\Pi_J:(\R^m)^d\mapsto\R^{n+1}\,,\quad\Pi_J(y)~\doteq~(y_{j_1},\ldots,y_{j_{n+1}})\,,
$$
and the linear operator
$$
A_J:\R^n\mapsto\R^{n+1}\,,\quad A_J(\lambda)~\doteq~\left(\lambda\cdot g(x_{\xi,j_1}), \ldots,\lambda\cdot g(x_{\xi,j_{n+1}})\right)\,.
$$
Then
$$
(\R^m)^d\backslash Y(\xi)~\subset~\bigcup_{\{J\subset\{0,\ldots,m-1\}^d~|~\#J=n+1\}}\big\{y\in(\R^m)^d~|~\Pi_J(y)\in A_J(\R^n)\big\}\,.
$$
For any $\xi\in[0,h]^d$ and $j\in\{0,\ldots,m-1\}^d$, define
$$
\tilde y_j(\xi)~\doteq~\tilde w(x_{j,\xi})\,.
$$
By the density of $Y(\xi)$ in $(\R^m)^d$, we can find $y(\xi)\in Y(\xi)$ satisfying
\begin{equation}\label{yty}
\left|y_j(\xi)-\tilde y_j(\xi)\right|~<~{\eta\over2}\qquad\forall\,j\in\{0,\ldots,m-1\}^d\,.
\end{equation}
On the other hand, fixed any $\xi\in [0,h]^d $ and $\lambda \in [-N,N]^n$, there exist  $r_{\lambda},\delta_{\lambda}>0$ such that 
$$
\inf_{\lambda'\in B(\lambda,r_{\lambda})}~\left|y_j(\xi)-\lambda'\cdot g(x_{j,\xi})\right|~>~\delta_{\lambda}\qquad\forall j\in \{0,\ldots,m-1\}^d\backslash J_{\lambda,\xi}(y(\xi))\,.
$$
As in the previous step, let $\{\lambda^1,...,\lambda^M\} \subset [-N,N]^n$ be a finite family such that 
$$
[-N,N]^n~\subset~\bigcup^M_{k=1}B_n\Big(\lambda^k,r_{\lambda^k}\Big)\,.
$$
Set $\delta\doteq\min_{k\in\{1,2,...,M\}}\delta_{k}$. For any $\lambda \in [-N,N]^n$, there exists  an index $k\in\{1,\ldots,M\}$ such that 
$$
\left|y_j(\xi)-\lambda\cdot g(x_{j,\xi})\right|~>~\delta\qquad\forall j\in \{0,\ldots,m-1\}^d\backslash J_{\xi,\lambda^k}(y(\xi))\,.
$$
Thus, by the uniform continuity of $g$ and the uniformly bound of $\lambda$, there exists a neighborhood $\N(\xi)$ of $\xi$ (independent on $\lambda$) such that
$$
\left|y_j(\xi)-\lambda\cdot g(x_{j,\xi'})\right|~>~{\delta\over2}\qquad\forall j\in \{0,\ldots,m-1\}^d\backslash J_{\xi,\lambda^k}(y(\xi)), \xi'\in \N(\xi)\,.
$$
In particular, recalling (\ref{Jxila}), we obtain that
$$
J_{\xi',\lambda}(y(\xi))\subset J_{\xi,\lambda^k}(y(\xi))\qquad \forall\,\xi'\in\N(\xi)\,,
$$
and this yields
%$$
%y(\xi)\in Y(\xi')\qquad\forall\,\xi'\in\N(\xi)\,.
%$$
%Therefore,
\begin{equation}\label{Nxi}
\#\,J_{\lambda,\xi'}(y(\xi))~\leq~n \qquad\forall\,\lambda\in [-N,N]^n,\ \forall\,\xi'\in\N(\xi)\,.
\end{equation}
Cover the set $[0,h[^d$ with finitely many disjoint neighborhoods $\{\N(\xi_k)\}_{k=1,\ldots,\ell}$ and define a piecewise constant function $w:[0,1[^d\,\mapsto \R$ by setting
$$w(x)~\doteq~y_j(\xi_k)\qquad\hbox{if}\qquad x \in \N(\xi_k) + h\,j\,,\quad k=1,\ldots,\ell\,,\quad j\in\{0,\ldots, m-1\}^d\,.$$
For any $x\in [0,1[^d$, let $k\in\{1,\ldots,\ell\}$ and $j\in\{0,\ldots, m-1\}^d$ be such that $x\in\N(\xi_k)+h\,j$. Then, $x$ and $x_{j,\xi_k}$ belong to $[0,h[^d+h\,j$. Recalling (\ref{ucont}) and (\ref{yty}), we have
\[
|w(x)-\tilde w(x)|~\leq~|y_j(\xi_k)-\tilde y_j(\xi_k)|~+~|\tilde w(x_{\xi_k,j})-\tilde w(x)|~<~\eta
\]
and it yields (\ref{ball}).
\quad\\
\quad\\
Moreover, by (\ref{msmall}), (\ref{Jxila}) and (\ref{Nxi}), we obtain
\begin{eqnarray*}
\hbox{meas}\,\Big(\big\{x\in K&|& w(x)=\lambda\cdot g(x)\big\}\Big)\\
&=&~\hbox{meas}\,\left(\bigcup_{j\in\{0,\ldots, m-1\}^d}\left\{x\in [0,h[^d+h\,j~|~ w(x)=\lambda\cdot g(x)\right\}\right)\\
&=&~\hbox{meas}\,\left(\bigcup_{j\in\{0,\ldots, m-1\}^d}\bigcup_{k=1}^\ell\left\{x\in \N(\xi_k)+h\,j~|~ y_j(\xi_k)=\lambda\cdot g(x)\right\}\right)\\
&\leq&~\sum_{k=1}^\ell\hbox{meas}\,\left(\bigcup_{j\in\{0,\ldots, m-1\}^d}\left\{\xi'\in \N(\xi_k)~|~ y_j(\xi_k)=\lambda\cdot g(x_{j,\xi'})\right\}\right)\\
&\leq&~\sum_{k=1}^\ell n\cdot\hbox{meas}\,\Big( \N(\xi_k)\Big)~=~n\,h^d~<~{\ve\over2}\,
\end{eqnarray*}
for every $\lambda\in[-N,N]^n$.
\quad\\
\quad\\
Finally, by Lusin's theorem, we then modify $w$ on a set of measure $<\ve/2$ and make it continuous
on the entire set $K$ and still satisfying (\ref{ball}). Then $w\in \W^g_{\ve,N}\cap B(\tilde w,\eta)$ and the set $\W^g_{\ve,N}$ is dense in $\C(K,\R)$.
\qed
\end{proof}
We are now going to prove our main theorem.
\quad\\
\quad\\
{\bf Proof of Theorem \ref{MT1}.} It is divided into 2 steps:
\quad\\
{\bf 1.} Fix $v=(v_1,\ldots,v_m)\in \C(K,\R^m)$ and let $\theta^*=(\theta^*_1,\ldots,\theta^*_m)$ be a solution of the optimization problem (\ref{P1})--(\ref{C1}). We claim that if $\theta^*$ is not extremal, then it is not the unique solution of (\ref{P1})--(\ref{C1}) and there exist two indexes $i_1\neq i_2\in\{1,\ldots,m\}$ and a Lagrange multiplier $\lambda
=(\lambda_1,\ldots,\lambda_n) \in\R^n$ satisfying
\begin{equation}\label{OC4}
\hbox{meas}\Big(\left\{x\in K~|~v_{i_1}(x)-v_{i_2}(x)=\lambda\cdot\big(f_{i_1}(x)-f_{i_2}(x)\big)\right\}\Big)~>~0\,.
\end{equation}
Indeed, if $\theta^*$ is non-extremal then the set 
\[
K_1~=~\left\{x\in K~|~0<\theta^*_i(x)<1~\mathrm{for~some}~i\in\{1,\dots,m\}\right\}
\]
has a positive Lebesgue measure. Since $\sum_{i}^{m}\theta^*(x)=1$ for all $x\in K$, we can deduce that  there exist two different indexes  $i_1, i_2\in\{1,\ldots,m\}$ such that
\[\hbox{meas}\big(\big\{x\in K~|~ 0<\theta^*_i(x)<1\,,\ \forall\,i\in\{i_1,i_2\}\big\}\big)~>~0\,.\]
Observe that 
\begin{eqnarray*}
\hbox{meas}\big(\big\{x\in K&|& 0<\theta^*_i(x)<1\,,\ \forall\,i\in\{i_1,i_2\}\big\}\big)\\
&=&~\hbox{meas}~\left(\bigcup_{n=3}^{+\infty}\left\{x\in K~\Big|~ {1\over n}<\theta^*_i(x)<1-{1\over n}\,,\ \forall\,i\in\{i_1,i_2\}\right\}\right)\,,
\end{eqnarray*}
there exists $n_0\geq 3$ such that the set 
\[
\tilde K~=~\left\{x\in K~\Big|~ {1\over n_0}<\theta^*_i(x)<1-{1\over n_0}\,,\ \forall\,i\in\{i_1,i_2\}\right\}
\]
has a positive Lebesgue measure. 
\quad\\
\quad\\
Consider the auxiliary optimization problem 
\begin{equation}\label{P4}
\displaystyle \mathrm{Minimize}_{\xi\in\mathcal{A}_0}~\int_{\tilde K}\xi(x)\big(v_{i_1}(x)-v_{i_2}(x)\big)~dx\,,
\end{equation}
where 
\begin{equation}\label{C4}
\mathcal{A}_0~\doteq~\left\{\xi\in\M(\tilde K,[-1,1])~\Big|~\int_{\tilde K}\xi(x)\big(f_{i_1}(x)-f_{i_2}(x)\big)~dx~=~0\right\}\,.
\end{equation}
Observe that $\xi^*\equiv 0$ is an optimal solution of (\ref{P4}) - (\ref{C4}). Indeed, for any $\xi\in\mathcal{A}_0$, define the mapping $\tilde{\theta}:K\mapsto\R^m$ by 
$$
\tilde{\theta}(x)~\doteq~\left\{\begin{array}{llll}\theta^*(x)+ {1\over n_0}\xi(x)\big({\bf{e}}_{i_1}-{\bf{e}}_{i_2}\big) & \ \ \hbox{if }x\in\tilde K\\
\theta^*(x) & \ \ \hbox{if }x\in K\setminus \tilde K\,,\end{array}\right.
$$
where $\{{\bf{e}}_1,\ldots,{\bf{e}}_m\}$ is the canonical basis of $\R^m$.
Clearly, $\tilde{\theta}$ belongs to $\mathcal{A}_{\alpha}$. Thus, 
$$
\int_{K}\tilde{\theta}(x)\cdot v(x)dx~\geq~\int_{K}\theta^*(x)\cdot v(x)dx
$$
and it implies that 
\begin{equation}\label{cd4}
\int_{\tilde K}\xi(x)\big(v_{i_1}(x)-v_{i_2}(x)\big)~dx~\geq~0\,.
\end{equation}
Now let's consider the vector subspace $Y$ of $\R^n$ generated by
$$
\left\{\int_{\tilde K}\xi(x)\big(f_{i_1}(x)-f_{i_2}(x)\big)~dx~\Big|~\xi\in\M(\tilde K,[-1,1])\right\}
$$
and define two convex subsets of $\R\times Y$
$$
A~\doteq~\left\{(a_0,0)\in\R\times Y~\big|~a_0<0\right\}\,,
$$
and $B$ the set of elements of the form
$$
(b_0,\bar{b})~=~\Big(\int_{\tilde K}\xi(x)\big(v_{i_1}(x)-v_{i_2}(x)\big)~dx\,,\,\int_{\tilde K}\xi(x)\big(f_{i_1}(x)-f_{i_2}(x)\big)~dx\Big)\,,
$$
with $ \xi$ varying in $\M(\tilde K,[-1,1])$.
Recalling (\ref{cd4}), one has that $A\cap B=\emptyset$.  Thanks to hyperplane separation theorem, there exists $(\lambda_0,\bar{\lambda})\in\big([0,+\infty)\times Y\big)\setminus\{(0,0)\}$ such that 
$$
\lambda_0 a_0~\leq~\lambda_0 b_0+\bar{\lambda}\cdot \bar{b}\qquad\forall \,a_0<0, (b_0,\bar{b})\in B\,.
$$
Observe that $\lambda_0\neq0$, otherwise we have 
$$
\bar\lambda\cdot \int_{\tilde K} \xi(x)\big(f_{i_1}(x)-f_{i_2}(x)\big)~dx~\geq~0\qquad\forall \,\xi\in\M(\tilde K,[-1,1])\,,
$$
that is impossible, since $0\neq\bar\lambda\in Y$. Setting $\lambda=-\bar\lambda/\lambda_0$, we obtain
$$
\int_{\tilde K}\xi(x)\big(v_{i_1}(x)-v_{i_2}(x)\big)~dx-\lambda\cdot \int_{\tilde K} \xi(x)\big(f_{i_1}(x)-f_{i_2}(x)\big)~dx~\geq~\lim_{a_0\to 0-}a_0~=~0
$$
for every $\xi\in\M(\tilde K,[-1,1])$.
This yields
$$
v_{i_1}(x)-v_{i_2}(x)~=~\lambda\cdot \big(f_{i_1}(x)-f_{i_2}(x)\big) \qquad \hbox{a.e.}~x\in \tilde K
$$
and consequently (\ref{OC4}).\\
\quad\\
In order to see that $\theta^*$ is not the unique solution of (\ref{P1})--(\ref{C1}), consider a function $\xi\in\mathcal{A}_0$ such that
\[
\hbox{meas}~\left(\left\{x\in \tilde K~\Big|~\xi(x)\neq 0\right\}\right)~>~0\,.
\]
Therefore, the following mappings
$$
\tilde{\theta}^{\pm}(x)~\doteq~\left\{\begin{array}{llll}\theta^*(x)\pm {1\over n_0}\xi(x)\big({\bf{e}}_{i_1}-{\bf{e}}_{i_2}\big) & \ \ \hbox{if }x\in\tilde K\\
\theta^*(x) & \ \ \hbox{if }x\in K\setminus \tilde K\end{array}\right.
$$
belong to $\mathcal{A}_{\alpha}$, satisfy $\tilde{\theta}^+\not\equiv\tilde{\theta}^-$ and
\[
\min~\left\{\int_{K}\tilde{\theta}^-(x)\cdot v(x)dx, \int_{K}\tilde{\theta}^+(x)\cdot v(x)dx\right\}~\leq~\int_{K}\theta^*(x)\cdot v(x)dx\,.
\]
\quad\\
{\bf 2.} Remark that if the problem (\ref{P1})--(\ref{C1}) admits two distinct solutions $\theta^*$ and $\theta^{**}$, then their convex combination
$${\tilde\theta}~\doteq~{\theta^*+\theta^{**}\over2}$$
is still a solution and it is not extremal. Therefore, by the previous step, $\mathcal{V}_{\alpha}$ contains the set of functions $v=(v_1,\ldots,v_m)\in\C(K,\R^m)$ satisfying
$$\hbox{meas}\,\Big(\big\{x\in K~|~v_{i_1}(x)-v_{i_2}(x)=\lambda\cdot \big(f_{i_1}(x)-f_{i_2}(x)\big)\big\}\Big)~=~0\quad\forall\,i_1\neq i_2,\,\lambda\in \R^n\,.$$
For any Lebesgue integrable function $g: K\to\R^n$, define $\mathcal{W}^{g}$ as in the statement of Lemma \ref{lemma}.
We then have 
\[
\mathcal{V}_{\alpha}~\supset~\bigcap_{i_1\neq i_2\in \{1,\ldots,m\}}~\left\{v=(v_1\ldots,v_m)\in\C(K,\R^m)~\Big|~v_{i_1}-v_{i_2}\in\W^{f_{i_1}-f_{i_2}}\right\}\,.
\]
By Lemma \ref{lemma}, the set $\mathcal{W}^{f_{i_1}-f_{i_2}}$ is residual in $\C(K,\,\R)$ for all $i_1\neq i_2 \in \{1,2,...,m\}$, i.e., there exists a family of open and dense subsets $\left\{\mathcal{W}^{f_{i_1}-f_{i_2}}_k\right\}_{k\in I\!\!N}$ of $\C(K,\R)$ satisfying
$$\bigcap_{k\in I\!\!N}\W^{f_{i_1}-f_{i_2}}_k~\subset~\W^{f_{i_1}-f_{i_2}}\,.$$
Hence we obtain
\begin{eqnarray*}
\mathcal{V}_{\alpha}&\supset&\bigcap_{i_1\neq i_2\in \{1,\ldots,m\}}\Big\{v\in\C(K,\R^m)~\Big|~v_{i_1}-v_{i_2}\in\bigcap_{k\in I\!\!N}\W_k^{f_{i_1}-f_{i_2}}\Big\}\cr\cr
&\supset&\bigcap_{i_1\neq i_2\in \{1,\ldots,m\}\,,\,k\in I\!\!N}\Big\{v\in\C(K,\R^m)~\Big|~v_{i_1}-v_{i_2}\in\W_k^{f_{i_1}-f_{i_2}}\Big\}\,.
\end{eqnarray*}
Moreover, it is not difficult to verify that the sets of the last intersection are open and dense. Therefore we can conclude that $\V_{\alpha}$ contains the intersection of countably many open dense subsets of $\C(K,\R^m)$, i.e. it is residual.
\qed
%\end{proof}
\quad\\
\quad\\

With similar techniques we can deal with a countable family of integrable functions. Let $(f_i)_{i=1}^\infty$ be a family of Lebesgue integrable functions from $K\subset \R^d$ to $\R^n$ satisfying
\begin{equation}\label{Hp}
\int_K \sup_{i}\|f_i(x)\|\, dx~<~\infty\,,
\end{equation}
where $\|\cdot\|$ is the norm in $\R^n$.
Let $(\bar\theta_i)_{i=1}^\infty$ be a family of measurable functions from $K$ to $[0,+\infty)$ such that
$$\sum_{i=1}^\infty\bar\theta_i(x)=1\qquad \forall\,x\in K\,.$$
We can consider $\bar \theta=(\bar\theta_i)_{i=1}^\infty$ as an element of the space ${\bf L}^\infty(K,\ell^\infty)$, where $\ell^\infty$ is the space of bounded real sequences. 
Call
$$\alpha~\doteq~\int_K\sum_{i=1}^\infty \bar\theta_i(x)\, f_i(x)\, dx\,.$$
Thanks to (\ref{Hp}) and dominated convergence, $\alpha\in\R^n$. Given $v\in \C(K,\ell^1)$, consider the problem
\begin{equation}\label{P3}
\displaystyle \mathrm{Minimize}_{\theta\in\mathcal{A}_{\alpha}}~\int_{K}\sum_{i=1}^\infty\theta_i(x)v_i(x)dx
\end{equation} 
over the set
\begin{eqnarray}\label{C3}
\mathcal{A}_{\alpha}~\doteq~\Big\{\theta\in {\bf L}^\infty(K,\ell^\infty)~&\Big|&~\theta_i(x)\geq0\ \ \forall\,i\in I\!\!N\,,\sum_{i=1}^\infty\theta_i(x)=1,\ a.e.\ x\in K,\qquad\nonumber\\
&&\qquad\quad\quad\qquad\int_{K}\sum_{i=1}^\infty\theta_i(x) f_i(x)~dx~=~\alpha\Big\}\,.
\end{eqnarray}
This problem admits at least a solution, since it is equivalent to 
$$\label{P3-1}
\displaystyle \mathrm{Minimize}_{\tilde{\theta}\in\mathcal{B}}~\int_{K}\sum_{i=1}^{\infty}\tilde{\theta}_i(x)\, \big(v_{i+1}(x)-v_1(x)\big)dx
$$
over the set
\begin{eqnarray*}\label{C3-1}
\mathcal{B}~\doteq~\Big\{\tilde{\theta}&\in&{\bf L}^{\infty}(K,\ell^\infty)~\Big|~\tilde{\theta}_i(x)\geq0\ \ \forall\,i\in I\!\!N\,,\ \sum_{i=1}^{\infty}\tilde{\theta}_i(x)\leq1,\ a.e.\ x\in K,\nonumber\\
&&\qquad\quad\qquad\sum_{i=1}^{\infty}\int_{K}\tilde{\theta}_i(x) \big(f_{i+1}(x)-f_1(x)\big)~dx~=~\alpha-\int_{K}f_1(x)~dx\Big\}
\end{eqnarray*}
and $\mathcal{B}$ is weakly*-compact in ${\bf L}^{\infty}(K,\ell^\infty)$.
\begin{theorem}\label{MT2}
Assume (\ref{Hp}). Then the set
\begin{equation}\label{W3}
\mathcal{V}_{\alpha}~\doteq~\left\{ v\in \C(K,\ell^1)~|~ (\ref{P3})-(\ref{C3})~has~a~unique~solution\right\}\,.
\end{equation}
 is residual in $\C(K,\ell^1)$. Moreover, for any $v\in\mathcal{V}_{\alpha}$, the unique optimal solution $\theta^*$ verifies $\theta^*_i(x)\in\{0,1\}$ for almost every $x\in K$ and every $i$.
\end{theorem}
\begin{proof}
The proof is similar to the one of Theorem \ref{MT1}.
Fix $v\in \C(K,\ell^1)$ and let $\theta^*\in {\bf L}^\infty(K,\ell^\infty)$ be a solution of the optimization problem (\ref{P3})--(\ref{C3}). If $\theta^*$ does not verify $\theta^*_i(x)\in\{0,1\}$ for almost every $x\in K$ and every $i$, then it is possible to show as above that $\theta^*$ is not the unique solution of (\ref{P3})--(\ref{C3}). We claim that there exist two indexes $i_1\neq i_2$ and $\lambda
=(\lambda_1,\ldots,\lambda_n) \in\R^n$ satisfying
\begin{equation}\label{OC5}
\hbox{meas}\Big(\left\{x\in K~|~v_{i_1}(x)-v_{i_2}(x)=\lambda\cdot\big(f_{i_1}(x)-f_{i_2}(x)\big)\right\}\Big)~>~0\,.
\end{equation}
Indeed, if $\theta^*$ is non-extremal, we have
\begin{eqnarray*}
0~&<&~\hbox{meas}\big(\big\{x\in K~|~0<\theta^*_i(x)<1~\mathrm{for~some}~i\big\}\\
%&=&~\hbox{meas}\big(\big\{x\in K~|~ 0<\theta^*_i(x)<1\,,\ \forall\,i\in\{i_1,i_2\}\big\}\big)\\
%&=&~\hbox{meas}\big(\big\{x\in K~|~ 0<\theta^*_i(x)<1\,,\ \forall\,i\in\{i_1,i_2\}\big\}\big)\\
&=&~\hbox{meas}~\left(\bigcup_{I\in I\!\!N}\bigcup_{n=3}^{+\infty}\left\{x\in K~\Big|~ {1\over n}<\theta^*_i(x)<1-{1\over n}\,,\ \forall\,i\in\{i_1,i_2\},\,\mathrm{some }~i_1\neq i_2\leq I\right\}\right)\,.
\end{eqnarray*}
Consequently, there exist $i_1\neq i_2$ and $n_0\geq 3$ such that the set 
\[
\tilde K~=~\left\{x\in K~\Big|~ {1\over n_0}<\theta^*_i(x)<1-{1\over n_0}\,,\ \forall\,i\in\{i_1,i_2\}\right\}
\]
has a positive Lebesgue measure. As in the proof of Theorem \ref{MT1}, one can verify that $\xi^*\equiv 0$ is an optimal solution of the auxiliary problem (\ref{P4}) - (\ref{C4}) and that it satisfies the necessary condition (\ref{OC5}) for some Lagrange multiplier $\lambda
=(\lambda_1,\ldots,\lambda_n) \in\R^n$.
Therefore, if we denote by $\mathcal{W}^{f_{i_1}-f_{i_2}}$ is the set of functions $w\in \C(K,\R)$ such that
$$
\hbox{meas}\Big(\bigl\{x\in K~|~ w(x)=\lambda\cdot (f_{i_1}(x)-f_{i_2}(x)) \bigr\}\Big)~=~0\qquad\hbox{for all}~~\lambda\in\R^n\,,
$$
we obtain
\[
\mathcal{V}_{\alpha}~\supset~\bigcap_{i_1\neq i_2}~\left\{v\in \C(K,\ell^1)~\Big|~v_{i_1}-v_{i_2}\in\W^{f_{i_1}-f_{i_2}}\right\}\,.
\]
By Lemma \ref{lemma}, for all $i_1\neq i_2$ the set $\mathcal{W}^{f_{i_1}-f_{i_2}}$ is residual in $\C(K,\R)$, i.e., there exists a family of open and dense subsets $\left\{\mathcal{W}^{f_{i_1}-f_{i_2}}_k\right\}_{k\in I\!\!N}$ of $\C(K,\R)$ satisfying
$$\bigcap_{k\in I\!\!N}\W^{f_{i_1}-f_{i_2}}_k~\subset~\W^{f_{i_1}-f_{i_2}}\,.$$
Hence we obtain
\begin{eqnarray*}
\mathcal{V}_{\alpha}&\supset&\bigcap_{i_1\neq i_2}\Big\{v\in \C(K,\ell^1)~\Big|~v_{i_1}-v_{i_2}\in\bigcap_{k\in I\!\!N}\W_k^{f_{i_1}-f_{i_2}}\Big\}\cr\cr
&\supset&\bigcap_{i_1\neq i_2\,,\,k\in I\!\!N}\Big\{v\in \C(K,\ell^1)~\Big|~v_{i_1}-v_{i_2}\in\W_k^{f_{i_1}-f_{i_2}}\Big\}\,.
\end{eqnarray*}
Consequently, $\V_\alpha$ is residual in $\C(K,\ell^1)$.
\qed
\end{proof}

{\bf Acknowledgments.} This work was partially supported by a grant from the Simons Foundation/SFARI (521811, NTK).

\end{document}